\documentclass[a4paper, 11pt]{article}

\usepackage{amssymb, amsmath, bbm, amsthm}
\usepackage{fullpage}
\usepackage{graphicx}
\usepackage{url}
\theoremstyle{plain}

\newtheorem{theorem}{Theorem}[section]

\newtheorem{corollary}[theorem]{Corollary}
\newtheorem{proposition}[theorem]{Proposition}
\newtheorem{example}[theorem]{Example}

\numberwithin{equation}{section}

\theoremstyle{definition}
\newtheorem{definition}[theorem]{Definition}

\theoremstyle{remark}
\newtheorem{remark}[theorem]{Remark}

\newcommand{\R}{{\mathbb R}}

\newcommand{\intr}{\mathrm{int}}

\newcommand{\F}{{\mathbb F}}

\newcommand{\Prob}{\mathop{\mathbb{P}}\nolimits}
\newcommand{\E}{\mathop{\mathbb{E}}\nolimits}

\newcommand{\mmu}{\mathop {\bar{x}_{\mu}}}

\newcommand{\supp}{\mbox{supp}}

\title{The Wronskian parameterizes the class of diffusions with a given distribution at a random time}

\author{Martin Klimmek\thanks{M.Klimmek@warwick.ac.uk}\\
Department of Statistics, University of Warwick}

\date{\today}
\begin{document}

\maketitle

\let\thefootnote\relax\footnotetext{The author would like to thank David Hobson
for encouragement and helpful suggestions and Paavo Salminen and Yang Yuxin for helpful suggestions. Any errors are the author's.}

\begin{abstract}
We provide a complete characterization of the class of one-dimensional time-homogeneous diffusions consistent with a given law at an exponentially distributed time using classical results
in diffusion theory. To illustrate we characterize the class of diffusions with the same distribution as Brownian motion at an exponentially distributed time. 
\end{abstract}

\section{Introduction}
The aim of this article is to characterize the class of one-dimensional time-homogeneous diffusions with a given law at an exponentially distributed time. We show, for instance, that there is a one-parameter family of diffusion processes started at $0$ with the same law as Brownian motion at an exponentially distributed time. In general, given a probability distribution we find that consistent diffusions are parameterized by a choice of starting point and secondly by a choice of Wronskian. 

We use classical results due to Dynkin \cite{Dynkin} and Salminen \cite{Salminen} involving the $h$-transform (or Doob's $h$-transform) of a diffusion to provide necessary and sufficient conditions for a diffusion to have a given distribution at a random time. Previously, Cox, Hobson and Obloj  \cite{CoxHobsonObloj:2011} proved the existence of consistent diffusions when the first moment is finite. We recover the construction in \cite{CoxHobsonObloj:2011} as a canonical choice from the class of consistent diffusions.

The problem of constructing diffusions with a given law at a random time can be motivated as an inverse problem in finance. Suppose we are given European call option prices with a fixed expiry time for a continuum of strikes. It is a natural inverse problem to ask whether there exist models for the asset price process consistent with these call prices and desirable properties. This article is related to a particular solution of the inverse problem proposed by Carr and Cousot \cite{CarrCousot}. The idea in \cite{CarrCousot} is to construct gamma-subordinated martingale diffusions to fit call prices after scaling the gamma-subordinator to be exponentially distributed at the expiry time.  

It is well known (see Breeden and Litzenberger \cite{BreedenLitzenberger:78}) that knowledge of call prices is equivalent to knowledge of the marginal law of the underlying asset. The problem of constructing a gamma-subordinated martingale diffusion consistent with call options at a single maturity is therefore a version of the more general problem considered here. In this article we do not assume that the consistent process is a martingale, nor that the target distribution is centred and we recover a parameterized class of consistent diffusions rather than a particular consistent martingale diffusion.

The analogue problem of constructing diffusions with a given distribution at a deterministic time is considered by Ekstr{\"o}m, et. al. in 
\cite{EkHob:2011}. For a more general view on inverse problems see also Hobson \cite{Hobson:10}. 

Finally, this article is related to the inverse problem of constructing diffusions consistent with prices for perpetual American options or, more generally, with given value functions for perpetual horizon stopping problems, see Hobson and Klimmek \cite{HobsonKlimmek10}. As in this article, the underlying key idea in \cite{HobsonKlimmek10} is to construct consistent diffusions through the speed measure via the eigenfunctions of the diffusion.  

\section{Generalized diffusions and the $h$-transform}
Let $I \subseteq \R$ be a finite or infinite interval with a left endpoint $a$ and right endpoint 
$b$. Let $m$ be a non-negative, non-zero Borel measure on $\R$ with $I=\supp(m)$. Let $s: I \rightarrow \R$ be a strictly increasing and continuous function. Let  $B=(B_t)_{t \geq 0}$ be a Brownian motion started at $B_0=s(X_0)$ supported on a a filtration $\F^B=({\mathcal F}_u^B)_{u\geq 0}$ with local time process
$\{ L_u^z ; u \geq 0, z \in \R \}$. Define $\Gamma$ to be the
continuous, increasing, additive functional
\[\Gamma_u = \int_{\R} L_u^z m(dz),\]
and define its right-continuous inverse by
\[A_t = \inf \{u : \Gamma_u > t \}. \]
If $X_t = s^{-1}(B(A_t))$ then $X=(X_t)_{t \geq 0}$ is a one-dimensional regular diffusion with speed measure $m$ and scale function $s$ and $X_t \in I$ almost surely for all $t \geq 0$. 

Let $H_x=\inf\{u:X_u=x\}$. Then for $\lambda>0$ (see e.g. \cite{Salminen}),
\begin{equation} \label{eq:eigenfunction}
\xi_\lambda(x,y)=\E_{x}[e^{-\lambda H_y}]= \left\{\begin{array}{ll}
\frac{\varphi_\lambda(x)}{\varphi_\lambda(y)}  &\; x \leq y \\
\frac{\phi_\lambda(x)}{\phi_\lambda(y)} &\; x \geq y ,
\end{array}\right.
\end{equation}
where $\varphi_\lambda$ and $\phi_\lambda$ are respectively a strictly increasing and a strictly decreasing solution to the differential equation
\begin{equation} \label{eq:differentialsc}
\frac{1}{2} \frac{d}{dm} \frac{d}{ds} f = \lambda f.
\end{equation}
The two solutions are linearly independent with Wronskian 
$W_\lambda=\varphi_\lambda' \phi_\lambda-\phi_\lambda' \varphi_\lambda > 0$.
Recall that if a diffusion $X=(X_t)_{t \geq 0}$ is in natural scale, then the Wronskian $W_\lambda$ is a constant. In the smooth case, when $m$ has a density $\nu$ so that $m(dx)= \nu(x)dx$ and $s''$ is continuous,
(\ref{eq:differentialsc}) is equivalent to
\begin{equation} \label{eq:differentialsnice}
\frac{1}{2} \sigma^2(x) f''(x) + \alpha(x) f'(x) = \lambda f(x),
\end{equation}
where \[ \nu (x)=  \sigma^{-2}(x) e^{M(x)}, \ \ s'(x)=e^{-M(x)}, \  \
M(x)=\int_{0-}^x 2 \sigma^{-2} (z) \alpha(z) dz. \]

We will call the solutions to (\ref{eq:differentialsc}) the $\lambda$-eigenfunctions of the diffusion. We will scale the $\lambda$-eigenfunctions so that $\varphi_\lambda(X_0)=\phi_\lambda(X_0)=1$. 

The $\lambda$-eigenfunctions are well known to be  $\lambda$-excessive. We recall that a Borel-measurable function $h: I \rightarrow \R^+$ is 
$\lambda$-excessive if for all $x \in I$ and $t \geq 0$,
 $\E_x[e^{-\lambda t} h(X_t)] \leq h(x)$ and if $\E_x[e^{-\lambda t} h(X_t)] \rightarrow h(x)$ pointwise as $t \rightarrow 0$.

\begin{definition} \label{d:htransform}
Let $h$ be a $\lambda$-excessive function. The {\it h-transform} of a diffusion $X=(X_t)_{t \geq 0}$ is the diffusion $X^h=(X^h_t)_{t \geq 0}$ with transition function
\[P^h(t; x, A)=\frac{1}{h(x)} \int_A e^{-\lambda t} p(t;x,y) h(y) m(dy),\]
where $p$ is the transition density of $X$.
\end{definition}

By the following result due to Dynkin \cite{Dynkin} (see also Salminen \cite{Salminen} (3.1)), any diffusion $X$ can be transformed into a diffusion with a given marginal at an exponential killing time. Fix $\lambda > 0$ and let $X=(X_t)_{t \geq 0}$ be a diffusion with $\lambda$-eigenfunctions $\varphi_\lambda$ and $\phi_\lambda$. Let $T$ be an exponentially distributed random variable with parameter $\lambda$, independent of $X$. 

\begin{theorem} \label{t:Dynkin}
Given a probability measure $\mu$ on $[a,b]$ let 
\begin{equation} \label{eq:hrep}
h(x)=\int_{[a,b]} \frac{\xi_\lambda(x,y)}{\xi_\lambda(X_0,y)} \mu(dy).
\end{equation}
Then $\Prob(X_T^h \in dx)=\mu(dx)$. 
Conversely, let $h$ be a $\lambda$-excessive function with $h(X_0)=1$ and let $\gamma_X^h(dy)=\Prob(X_T^h \in dx)$. Then $h$ has the representation (\ref{eq:hrep})
with $\mu=\gamma_X^h$.
\end{theorem}

The measure $\gamma_X^h$ in (\ref{eq:hrep}) is called {\it the representing measure for $h$}. It follows from Theorem \ref{t:Dynkin} that we can start with any diffusion $X$ on $[a,b]$ and construct a killed diffusion with a given representing measure via an $h$-transform. Thus, in principle, since the representing measure co-incides with the law of $X^h_T$, Dynkin's result solves the inverse problem of constructing diffusions with a given marginal at an exponentially distributed (killing) time. 

We will build on this observation to recover consistent diffusions using a characterization of a representing measure in terms of the $\lambda$-eigenfunctions.

\section{Characterizing consistent diffusions}
Without loss of generality, we will restrict the inverse problem to the class of diffusions in natural scale. If it is possible to construct a diffusion in natural scale consistent with a given law on $[a,b]$ at an exponential time then it follows that we can also construct consistent diffusions with non-trivial scale functions: Given a marginal density $\mu$ on $\R$, define a measure $\nu$ via $\nu(B)=\mu(s^{-1}(B))$ for Borel sets $B \subseteq \R$, where $s:[a,b]\rightarrow \R$ is an arbitrary scale function. Then if a diffusion $Y=(Y_t)_{t\geq0}$ in natural scale is consistent with $\nu$, the diffusion $X=s(Y)$ is consistent with $\mu$.

Recall the definition of the $h$-transform and observe that $h \equiv 1$ is a $\lambda$-excessive function for any $\lambda > 0$. 
We will call the $h$-transform corresponding to $h \equiv 1$ the $\lambda$-transform. The $\lambda$-transform of $X$ is equivalent to $X$ up to 
the exponential time $T \sim Exp(\lambda)$ at which time $X^1$ is killed, while $X$ remains on the state space $I$. Thus
\[X^1_t= \left\{\begin{array}{ll}
X_t  &\;  t \leq T \\
\Delta  &\; t > T,  
\end{array}\right.
\]
where $\Delta$ is the grave state of the killed diffusion $X^1$. Note that the transition density of $X^1$ is given by $q(t;x,y)=e^{-\lambda t} p(t; x,y)$, where $p(t;x,y)$ is the transition density of $X$. Other fundamental quantities are related similarly, for instance $\E_{X_0}[e^{-\lambda H_x}]=\Prob_{X_0}(H_x < T)=\Prob_{X_0}(\mbox{$X^1$ reaches $x$})$.  

We restate our inverse problem as follows. Given a probability measure $\mu$ on $[a,b]$, construct a diffusion $X=(X_t)_{t\geq0}$ such that for all $x \in [a,b]$
\begin{equation} \label{eq:1rep}
1=\int_{[a,b]} \frac{\xi_\lambda(x,y)}{\xi_\lambda(X_0,y)} \mu(dy),
\end{equation}
whence by Theorem \ref{t:Dynkin}, $X^1_T \sim \mu$. Since $X_{T} \equiv X^1_{T} \sim \mu$, the idea is to construct the class of consistent diffusions via the $\lambda$-eigenfunctions for which (\ref{eq:1rep}) holds. 

The following result is an elementary case ($h \equiv 1$) of Proposition (3.3) in Salminen \cite{Salminen}.
\begin{proposition} \label{p:Salminen}
Given a diffusion $X$, the representing measure $\gamma=\gamma^1_X$ is given by
\begin{eqnarray} \label{eq:gamma}
\gamma([a,x)) &=& \frac{\varphi_\lambda'(x-)}{W_\lambda}  \mbox{$\ \ \ $ for $a<x \leq X_0$}, \\ 
\gamma((x,b]) &=& \frac{-\phi_\lambda'(x+)}{W_\lambda}  \mbox{$\ \ \ $ for $X_0 \leq x <b$},
\end{eqnarray}
where $\varphi_\lambda$ ($\phi_\lambda$) are the increasing (decreasing) $\lambda$-eigenfunctions of $X$ and $W_\lambda$ is the Wronskian.
\end{proposition}

\begin{remark} \label{r:endpoint}
Suppose that $a$ is accessible and that $X_0=a$. Then the representing measure for $h=1$ is given by $\gamma((x,b]) =\frac{-\phi_\lambda'(x+)}{W_\lambda}  \mbox{$\ $ for $a \leq x <b$}$ and similarly if $X_0=b$, where $b$ is accessible. 
\end{remark}

The characterization of the representing measure given by Proposition \ref{p:Salminen} will be used to arrive at our main result. Suppose we are given a probability measure $\mu$ on $[a,b]$. Let $U_\mu(x)=\int_{[a,b]} |x-y| \mu(dy)$, $C_\mu(x)=\int_{[a,b]} (y-x)^+ \mu(dy)$ and $P_\mu(x)=\int_{[a,b]}(x-y)^+ \mu(dy)$. Let $X=(X_t)_{t \geq 0}$ be a one-dimensional diffusion in natural scale and let $T$ be an independent exponentially distributed random variable with parameter $\lambda>0$.

\begin{theorem} \label{t:main}
Suppose $X_0 \in (a,b)$. Then $X_T \sim \mu$ if and only if the speed measure of $X$ satisfies
\[
m(dx)= \left\{\begin{array}{ll}
\frac{1}{2\lambda} \frac{\mu(dx)}{P_\mu(x)-P_\mu(X_0)+1/W_\lambda}  &\;  a < x \leq X_0 \\
\frac{1}{2\lambda} \frac{\mu(dx)}{C_\mu(x)-C_\mu(X_0)+1/W_\lambda}  &\; X_0 \leq x < b.  
\end{array}\right.
\]
where $W_\lambda>0$ is the Wronskian of $X$.
\end{theorem}

\begin{proof}
Suppose first that $X_T \sim \mu$. Then since $X_T \equiv X^1_T$, by Theorem \ref{t:Dynkin} $\mu$ is the representing measure for $h \equiv 1$. Differentiating both sides of 
(\ref{eq:gamma}) we find that for all points $x$ such that $a < x \leq X_0$ and which are not atoms of $\mu$,
\[\mu(dx)= \frac{1}{W_\lambda} \varphi_\lambda''(x) dx.\]
(If $\mu$ has an atom at $x$ then $\mu(\{x\})=\frac{1}{W_\lambda} (\varphi_\lambda''(x+)-\varphi_\lambda''(x-))$. The case $x \geq X_0$ is similar, with $\phi_\lambda$ replacing
$\varphi_\lambda$. 

On the other hand, integrating the two sides of (\ref{eq:gamma}) we have
\begin{eqnarray*} 
P_\mu(x)+k_1 &=& \frac{\varphi_\lambda(x)}{W_\lambda}   \mbox{$\ \ \ $ for $x \leq X_0$}, \\ 
C_\mu(x)+k_2 &=& \frac{\phi_\lambda(x)}{W_\lambda}  \mbox{$\ \ \ $ for $x \geq X_0$},
\end{eqnarray*}
for constants $k_1,k_2 \in \R$. Now using the fact that 
$\varphi_\lambda(X_0)=\phi_\lambda(X_0)=1$ we find that $k_1=1/W_\lambda-P(X_0)$ and $k_2=1/W_\lambda-C(X_0)$. Since $\varphi_\lambda$ and $\phi_\lambda$ are the $\lambda$-eigenfunctions for $X$ and solutions to (\ref{eq:differentialsc}), the speed measure of $X$ satisfies
\[
m(dx)= \left\{\begin{array}{ll}
\frac{1}{2 \lambda} \frac{\varphi_\lambda''(x) dx}{\varphi_\lambda(x)} &\;  a < x \leq X_0 \\
\frac{1}{2 \lambda} \frac{\phi_\lambda''(x) dx}{\phi_\lambda(x)}  &\; X_0 \leq x < b.   \\
\end{array}\right.
\]

Substituting for $\varphi_\lambda$ and $\phi_\lambda$ we thus have
\[
m(dx)= \left\{\begin{array}{ll}
\frac{1}{2\lambda} \frac{\mu(dx)}{P_\mu(x)+ k_1} &\;  a < x \leq X_0 \\
\frac{1}{2\lambda} \frac{\mu(dx)}{C_\mu(x)+  k_2}  &\; X_0 \leq x < b   \\
\end{array}\right.
\]
as required.

Conversely suppose that $X$ has the given speed measure on $(a,b)$. Define a function $\eta: [a,b] \rightarrow \R^+$ as follows. Let $W_\lambda >0$ be the Wronskian associated with $X$ and set
\[
\eta(x)= \left\{\begin{array}{ll}
W_\lambda(P_\mu(x)-P_\mu(x_0))+1 &\;  a \leq x \leq X_0 \\
W_\lambda(C_\mu(x)-C_\mu(x_0))+1  &\; X_0 \leq x \leq b.   \\
\end{array}\right.
\]
Then $\eta$ solves (\ref{eq:differentialsc}) on the domain $(a,b)$ and we therefore have 
\[
\eta(x)= \left\{\begin{array}{ll}
\varphi_\lambda(x) &\;  a \leq x \leq X_0 \\
\phi_\lambda(x)  &\; X_0 \leq x \leq b.   \\
\end{array}\right.
\]
By Proposition \ref{p:Salminen} the representing measure for 
$h \equiv 1$ is given by

\begin{alignat*}{4}
\gamma([a,x)) &= \frac{\eta'(x-)}{W_\lambda} &=\mu([a,x)) & \mbox{$\ \ \ $ for $a<x \leq X_0$} \\ 
\gamma((x,b]) &= \frac{-\eta'(x+)}{W_\lambda} &= \mu((x,b]) & \mbox{$\ \ \ $ for $X_0 \leq x <b$},
\end{alignat*}
and it follows that $X_T \sim \mu$.
\end{proof}

\begin{remark}
If $X$ is started at an accessible end-point, $a$ say, then $X_T \sim \mu$ if and only if for all $x \in [a,b)$, $m(dx)=\frac{1}{2\lambda} \frac{\mu(dx)}{C_\mu(x)-C_\mu(a)+1/W_\lambda}.$ The case $X_0=b$ where $b$ is accessible is analogous. Compare Remark \ref{r:endpoint}.
\end{remark}

We have the following interpretation for the Wronskian.
\begin{corollary} \label{c:Wronskian}
If $X_T \sim \mu$ then the Wronskian satisfies
\[\left. \frac{W_\lambda}{2\lambda}= \frac{m(dz)}{\mu(dz)} \right|_{z=X_0}.\]
\end{corollary}

Intuition for Corollary \ref{c:Wronskian} is provided by the fact that 
$2/W_\lambda=\E_{X_0}[L^{X_0}_{A_T}]$ (see Lemma VI. 54.1 in Rogers and Williams \cite{rogers}).

\section{The Wronskian and the martingale property}

By Theorem \ref{t:main} the class of diffusions with a given starting point and a given law at an exponentially distributed time is parameterized by a choice of Wronskian. 
We will now see that when the first moment of the target law is finite there exists a unique consistent diffusion with the property that it is a martingale away from the boundary.

We will suppose for the remainder of this section that $\int_{[a,b]} |x| \mu(dx) < \infty$. Let $\mmu=\int_{[a,b]} x \mu(dx)$ and let us fix the starting point, $X_0=\mmu$.  Let $\tau \equiv \inf\{t \geq 0: X_t \notin \intr(I)\}$. Then $X^\tau=(X_{t \wedge \tau})_{t \geq 0}$ is a local martingale. In the following we will say that $X$ is a martingale diffusion whenever $X^\tau$ is a martingale. 

It follows from Theorem \ref{t:main} that $X$ is a diffusion with law $\mu$ at an exponentially distributed time $T \sim Exp(\lambda)$ if and only if \[m(dx)=\frac{1}{\lambda}\frac{\mu(dx)}{U_\mu(x)-|x-X_0|-2C_\mu(X_0)+2/W_\lambda}\]
for $a < x < b$. Given this formula for the consistent speed measures, the most natural choice of $W_\lambda$ is perhaps $W_\lambda=1/C_\mu(\mmu)$. Indeed, will see in Theorem \ref{t:local} below that this choice recovers the unique consistent martingale diffusion (it also recovers the construction in \cite{CoxHobsonObloj:2011}).

\begin{theorem} \label{t:local}
Suppose $X_0=\mmu$ and $a=-\infty$ or $b=\infty$. Then $X$ is a martingale diffusion consistent with $\mu$ if and only if $W_\lambda=1/C_\mu(\mmu)$. 
\end{theorem}

The author would like to thank David Hobson for providing the proof that $\int^\infty \frac{x C''_\mu(x)}{C_\mu(x)} dx = \infty$ which is used below. 

\begin{proof}
We suppose $b=\infty$, the case $a=\infty$ is analogous. Since $m$ is positive, $W_\lambda \geq 1/C_\mu(\mmu)$. Suppose $W_\lambda>1/C_\mu(\mmu)$ then \[m(dx)=\frac{1}{\lambda}\frac{\mu(dx)}{U_\mu(x)-|x-\mmu|+c}\] for some $c>0$ and  
$\displaystyle \lim_{x \uparrow \infty} \frac{m(dx)}{\mu(dx)}=1/\lambda c$.
Thus $\int^{\infty} |x| m(dx) \propto \int^{\infty} |x|\mu(dx)< \infty$. 
It follows from Theorem 1 in Kotani \cite{Kotani} that $X$ is not a martingale diffusion.

Conversely suppose that $W_\lambda=1/C_\mu(\mmu)$. We will show that 
$\int^\infty \frac{x C_\mu''(x)}{C(x)} dx = \infty$. Write $h(x)=\frac{x C''(x)}{2 C(x)}$. For fixed $y$ and $x>y$, let $D(x)=\E^x\left[\exp\left(-\int_0^{H_y} \frac{h(B_s)}{B_s} ds \right)\right]$. Note that $D(y)=1$ and $D$ is positive and decreasing. Let $M_t=\exp\left(-\int_0^t \frac{h(B_s)}{B_s}\right) D(B_t)$. Then 
$M=(M_{t \wedge {H_y}})_{t \geq 0}$ is a bounded martingale. In particular, by It{\^o}'s formula,
$\frac{1}{2} D''(B_s)=\frac{h(B_s)}{B_s} D(B_s)$, so that 
$D(x)=\frac{C_\mu(x)}{C_\mu(y)}$. It follows that $\displaystyle \lim_{x \rightarrow \infty} D(x) =0$ and that \[\lim_{\substack{x \uparrow \infty \\ B_0=x}} \int_0^{H_y} \frac{h(B_s)}{B_s} ds = \infty\] almost surely. Then we must have
\begin{eqnarray*}
\infty &=&\lim_{\substack{x \uparrow \infty \\ B_0=x}} \E\left[\int_0^{H_y}\frac{h(B_s)}{B_s} ds \right] \\
&=& \lim_{x \uparrow \infty} \left\{\int_y^x \frac{h(z)}{z} (z-y) dz + \int_x^\infty \frac{h(z)}{z}(x-y) dz\right\} \\
&=& \int_y^\infty \frac{h(z)}{z}(z-y) dz,
\end{eqnarray*}
and thus $\int^\infty h(z) dz = \infty$. By Theorem 1 in \cite{Kotani}, $X$ is a martingale diffusion.
\end{proof}

\begin{remark}
An alternative proof of Theorem \ref{t:local} is available using a result
in Hulley and Platen \cite{HulleyPlaten}. By Theorem 1.2 and Proposition 2.2 in \cite{HulleyPlaten}, $X$ is a martingale diffusion if and only if
$\displaystyle \lim_{x \uparrow \infty} \phi_\lambda(x) = 0$. Now recall that since
$X$ is consistent with $\mu$ we have $\phi_\lambda(x)=W_\mu C_\mu(x) - W_\mu C_\mu(X_0) +1$ for $x \geq X_0$. Clearly  $\displaystyle \lim_{x \uparrow \infty} \phi_\lambda(x) = 0$ if and only if $W_\mu=1/C_\mu(X_0)$. 
\end{remark}

\section{Examples}
\begin{example} \label{ex:brownian}
Let $B=(B_t)_{t \geq 0}$ be Brownian motion. 
Then we find that $B_T \sim \mu_\lambda$, where for $x>0$
\[\mu_\lambda((x,\infty))=\mu_\lambda((-\infty,-x))=\frac{1}{2} e^{-\sqrt{2 \lambda} x}.\]
Let us recover the class of consistent diffusions started at $X_0=0$ with the 
same law at an exponential time as Brownian motion. The consistent
diffusions have speed measures $m_W(x)=\nu_W(x)dx$, where 
\[\nu_W(x)= \frac{e^{-\sqrt{2 \lambda} |x|}}{e^{-\sqrt{2 \lambda} |x|}-\sqrt{\lambda/2}+2\lambda/W}.\]
The choice $W=1/C(0)=2 \sqrt{2 \lambda}$ corresponds to Brownian motion. Any choice of $W \in (0,1/C(0))$ corresponds to a strict local martingale diffusion with the
same marginal law.  
\begin{figure}[ht!]\label{Fig.1}
\begin{center}
\includegraphics[height=4cm,width=10cm]{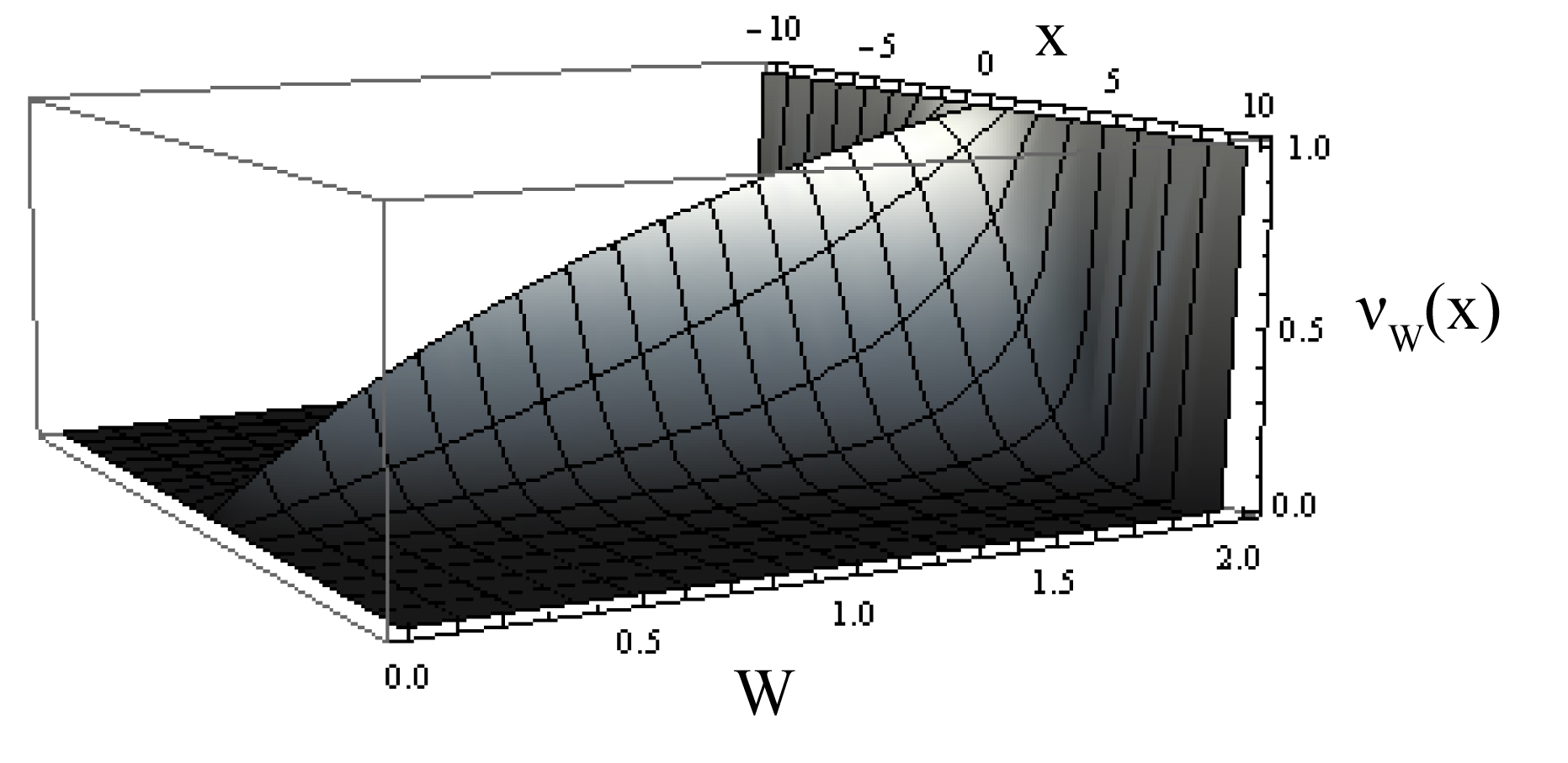}
\caption{Plot of $\nu_W(x)$ for $\lambda=1/2$ and $W \in (0,2]$. Note that 
$\nu_2(x) \equiv 1$ corresponds to Brownian motion which has Wronskian
$W=2\sqrt{2 \lambda}=2$.}
\end{center}
\end{figure}
\end{example}

\begin{example} \label{ex:uniform}
Suppose that $a=-1, b=1$ and we wish to recover diffusions started at $X_0=0$ that are uniformly distributed at an exponential time. We find that the consistent diffusions are parameterized by $W \in (0,4]$ with corresponding  speed measures $m_W(dx)=\nu_W(x)dx$ given by
\[
1/\nu_W(x)= \left\{\begin{array}{ll}
\lambda(x^2+2x+4/W) &\;  -1 \leq x \leq 0 \\
\lambda(x^2-2x+4/W)  &\; 0 \leq x \leq 1   \\
\infty &\; \mbox{otherwise}.
\end{array}\right.
\]

The canonical choice for $W$ is $1/W=C(0)=1/4$. Since $\nu_4(-1)=\nu_4(1)=\infty$ the boundary points are inaccessible whence $X^4_{T} \sim U(-1,1)$. For $W \in (0,4)$, the speed measure is finite on $[-1,1]$ and the consistent diffusions reflect at the boundaries and $X^W_{T} \sim U[-1,1]$.  

Now suppose $X_0=1/2$. Then
\[
1/\nu_W(x)= \left\{\begin{array}{ll}
\lambda(x^2+2x+1/4+4/W) &\;  -1 \leq x \leq 1/2 \\
\lambda(x^2-2x+9/4+4/W)  &\; 1/2 \leq x \leq 1   \\
\infty &\; \mbox{otherwise}.
\end{array}\right.
\]

\begin{figure}[ht!]\label{Fig.2}
\begin{center}
\includegraphics[height=5cm,width=8cm]{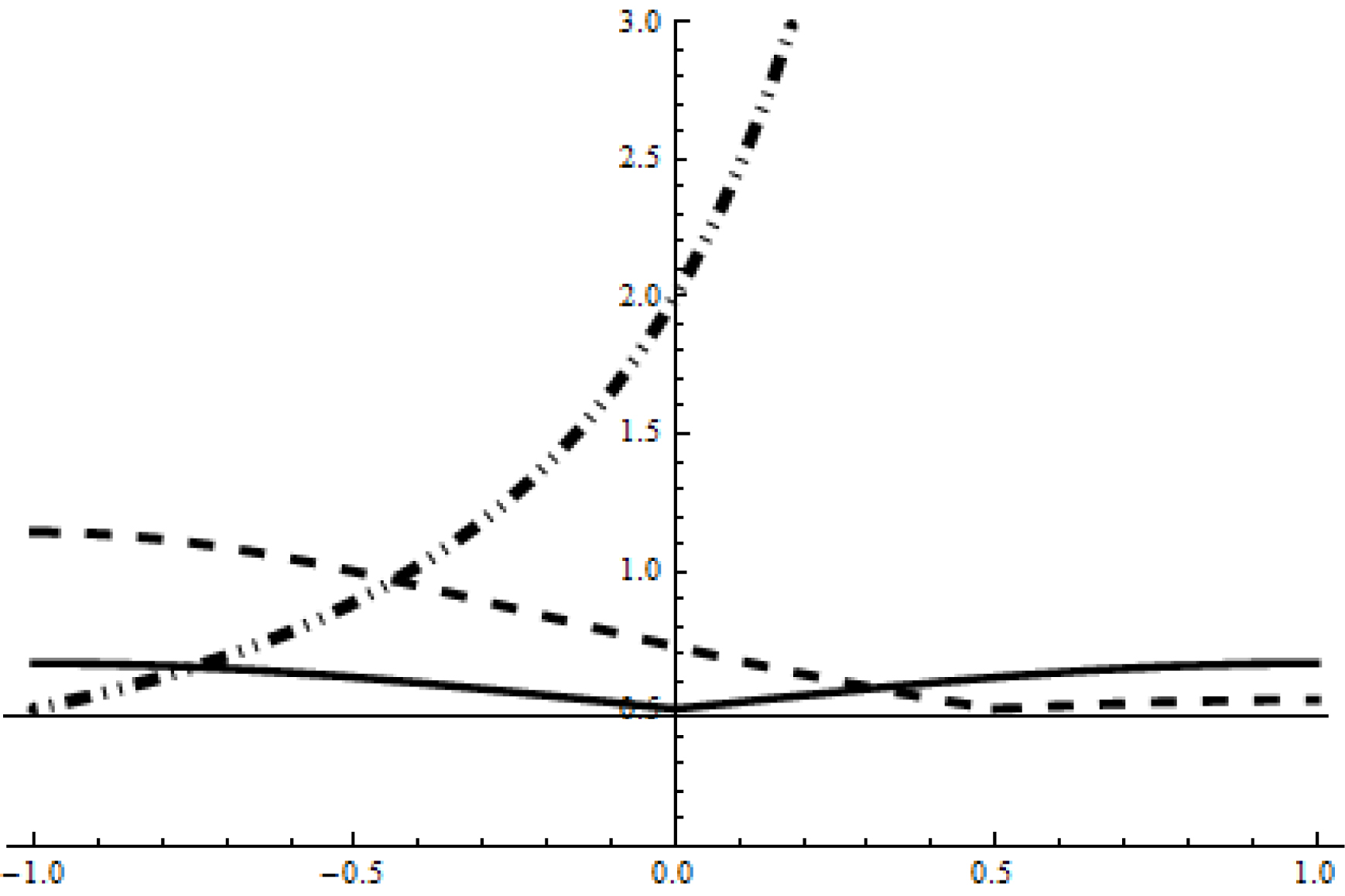}
\caption{Plot of $\nu_W(x)$ for $\lambda=1/2$ and $W=1$ when $X_0=0$ (solid line) and $X_0=1/2$ (dashed line), and $X_0=-1$ (alternating line). Note that 
$\nu_W(X_0)=\frac{W}{4\lambda}=1/2$, see Corollary \ref{c:Wronskian}.}
\end{center}
\end{figure}
\end{example}


\bibliography{biblio}

\end{document}